\documentclass[11pt]{article}
\usepackage{amsfonts}
\usepackage{amsmath}
\usepackage{geometry}
\usepackage{multicol}
\usepackage{graphicx}

\newtheorem{theorem}{Theorem}

\newtheorem{condition}[theorem]{Condition}

\newtheorem{definition}[theorem]{Definition}

\newtheorem{lemma}[theorem]{Lemma}

\newtheorem{proposition}[theorem]{Proposition}
\newtheorem{remark}[theorem]{Remark}

\newenvironment{proof}[1][Proof]{\noindent\textbf{#1.} }{\ \rule{0.5em}{0.5em}}
\input{tcilatex}
\begin{document}

\title{Noise prevents infinite stretching of the passive field in a
stochastic vector advection equation}
\author{F. Flandoli, M. Maurelli, M. Neklyudov}
\maketitle

\begin{abstract}
A linear stochastic vector advection equation is considered; the equation
may model a passive magnetic field in a random fluid. When the driving
velocity field is rough but deterministic, in particular just H\"{o}lder
continuous and bounded, one can construct examples of infinite stretching of
the passive field, arising from smooth initial conditions. The purpose of
the paper is to prove that infinite stretching is prevented if the driving
velocity field contains in addition a white noise component.
\end{abstract}

\section{Introduction}

Consider the linear stochastic vector advection equation in $\mathbb{R}^{3}$%
: 
\begin{equation}
d\mathbf{B}+\mathrm{curl}({\mathbf{v}}\times \mathbf{B})dt+\sigma
\sum_{k=1}^{3}\mathrm{curl}(\mathbf{e}_{k}\times \mathbf{B})\circ dW^{k}=0,
\label{vectadv}
\end{equation}%
where {$\mathbf{v}$}$:\left[ 0,T\right] \times \mathbb{R}^{3}\rightarrow 
\mathbb{R}^{3}$ is a given divergence-free vector field, the solution $%
\mathbf{B}$ is a divergence-free vector field, $\mathbf{e}_{1},\mathbf{e}%
_{2},\mathbf{e}_{3}$ is the canonical basis of $\mathbb{R}^{3}$, $\mathbf{W}%
=\left( W^{1},W^{2},W^{3}\right) $ is a Brownian motion in $\mathbb{R}^{3}$, 
$\sigma $ is a real number. The initial condition, at time $t=0$, will be
denoted by $\mathbf{B}_{0}$. The driving vector field (the velocity field of
the fluid, in the usual interpretation) is modeled by the Gaussian field%
\begin{equation*}
{\mathbf{v}}+\sigma \sum_{k=1}^{3}\mathbf{e}_{k}\frac{dW^{k}}{dt}={\mathbf{v}%
}+\sigma \frac{d\mathbf{W}}{dt}
\end{equation*}%
where {$\mathbf{v}$} is deterministic, a sort of average or slow-varying
component, and $\sigma d\mathbf{W}$\ is the fast-varying random component,
white noise in time. This equation may model a passive vector field $\mathbf{%
B}$, like a magnetic field, in a turbulent fluid with a non-trivial average
component {$\mathbf{v}$}. The intensity $\sigma $ of the noise can be
arbitrarily small, in the sequel, to model real situations when the noise
(which always exists) is usually neglected in first approximation. However,
the trajectories of $\mathbf{W}$\ are only H\"{o}lder continuous with
exponent smaller than $\frac{1}{2}$ and not differentiable at any point, so
that the impulses given by the term $\sigma \frac{d\mathbf{W}}{dt}$ are
small when cumulated in time ($\sigma \mathbf{W}$) but istantaneously very
strong. We aim at studying existence, uniqueness, representation formula and
regularity under low regularity assumption on {$\mathbf{v}$}.

The key point of this work is the fact that the noise prevents blow-up,
under assumptions on {$\mathbf{v}$} such that blow-up may occur in the
deterministic case. When $\sigma =0$, we give an example of H\"{o}lder
continuous vector field {$\mathbf{v}$} such that infinite values of $\mathbf{%
B}$ arise in finite time from a bounded continuous initial field $\mathbf{B}%
_{0}$;\ then we prove that H\"{o}lder continuity and boundedness of {$%
\mathbf{v}$} is sufficient, in the stochastic case ($\sigma \neq 0$), to
prove that continuous initial field $\mathbf{B}_{0}$ produces continuous
fields $\mathbf{B}_{t}$ for all $t\geq 0$. The singularity in the
deterministic case is associated to infinite stretching of $\mathbf{B}$;\
randomness prevents stretching to blow-up to infinity. Precisely, we prove
(see the notations below):

\begin{theorem}
\label{Thm 1}i) For $\sigma =0$, there exists {$\mathbf{v}$}$\in
C_{b}^{\alpha }(\mathbb{R}^{3};\mathbb{R}^{3})$ and $\mathbf{B}_{0}\in
C^{\infty }(\mathbb{R}^{3};\mathbb{R}^{3})$ such that $\sup_{\left\vert
x\right\vert \leq 1}\left\vert \mathbf{B}\left( t,x\right) \right\vert
=+\infty $ for all $t>0$.

ii) For $\sigma \neq 0$, for all {$\mathbf{v}$}$\in C([0,T];C_{b}^{\alpha }(%
\mathbb{R}^{3};\mathbb{R}^{3}))$ and $\mathbf{B}_{0}\in C(\mathbb{R}^{3};%
\mathbb{R}^{3})$ one has $\mathbf{B}\in C([0,T]\times \mathbb{R}^{3};\mathbb{%
R}^{3})$, with probability one.
\end{theorem}

Clearly, linear vector advection equation is a very idealized model of fluid
dynamics but this result opens the question whether noise may prevent
blow-up of the vorticity field of 3D Euler equations. The emergence of
singularity seems to require a certain degree of organization of the fluid
structures and perhaps this organization is lost, broken, under the
influence of randomness. With further degree of speculation, one could even
think that a turbulent regime may contain the necessary degree of randomness
to prevent blow-up; if so, singularities of the vorticity could more likely
be associated to strong transient phases, instead of established turbulent
ones.

From the mathematical side, this is not the first result of this nature, see 
\cite{FedriFla}, \cite{BFGM}, \cite{FGP-Eulero}, \cite{DelarueFlaVincenzi},
and also \cite{FlaGubPri}, \cite{Maure}, \cite{BFM PAMS}, \cite{Fla Saint
Flour}, \cite{GubiJara}, \cite{NevesOliv} for uniqueness of weak solutions
due to noise (the other face of the celebrated open problem presented by 
\cite{Fefferman}). However, these papers deal with scalar problems, like
linear transport equations, linear continuity equations, vorticity in 2D
Euler equations, 1D Vlasov-Poisson equations. The result of the present
paper is the first one dealing with vector valued PDEs like 3D Euler
equations;\ the kind of singularity in the vectorial case is different,
related to rotations and stretching instead of shocks or mass concentration.
Several new technical difficulties arise due to the vectorial nature of the
equation (for instance, the proof of uniqueness of non-regular solutions,
Lemma \ref{lemma uniqueness}, usually involving commutator estimates, here
is more difficult and is obtained by special cancellations, also inspired to 
\cite{NevesOliv}). Let us mention also the improvement of well-posedness due
to noise proved for dispersive equations, \cite{DT10}, \cite{ChoukGubi}. In
all the works mentioned so far the noise is multiplicative, and often of
transport type like in the present paper. The role of additive noise in
preventing singularities is more obscure. For uniqueness under poor drift,
additive noise is very powerful see \cite{V}, \cite{Zambotti}, \cite{Gyongy}%
, \cite{DaPrato Fla}, \cite{DFPR}; however, its relevance in fluid dynamics
is still under investigation. See \cite{DaPDeb}, \cite{FR Markov}, \cite%
{FlCime}, \cite{Romito} for partial results.

For additional details on vector advection equations see for instance \cite%
{BrzNekl} and references therein. For advanced results on the
differentiability of stochastic flow generated by rough drift (key
ingredient of the representation formula (\ref{reprform})), see \cite{AryPil}%
, \cite{BFGM}, \cite{FedriFla}, \cite{MohNilPro}. For a general reference on
passive advection driven by random velocity fields see \cite{FalGawVerg},
where also the case of a passive magnetic field is discussed; the structure
of the noise term in the present work is very simplified with respect to 
\cite{FalGawVerg} but the point here is to prove that noise has a depleting
effect on $\mathbf{B}$ and this fact is true also under this simple noise;
generalization to space-homogeneous noise with more complex space structure
is possible, if $Q\left( 0\right) $, the covariance matrix at $x=0$, is
non-degenerate.

The model described here is clearly too idealized for a direct interest in
fluid dynamics but once the phenomenon of depletion of stretching is
rigorously proved in this particular framework, there is more motivation to
investigate generalizations which could become closer to reality. One of
them would be the case when {$\mathbf{v}$} contains (also just small) high
frequency fluctuations, although not being white noise. This extension looks
very difficult but potentially not impossible.

\subsection{Notations}

We denote by $C(\mathbb{R}^{3};\mathbb{R}^{3})$ (resp. $C^{\infty }(\mathbb{R%
}^{3};\mathbb{R}^{3})$) the space of all continuous (resp. infinitely
differentiable) vector fields $\mathbf{v}:\mathbb{R}^{3}\rightarrow \mathbb{R%
}^{3}$. We denote by $C_{b}(\mathbb{R}^{3};\mathbb{R}^{3})$ the space of all 
$\mathbf{v}\in C(\mathbb{R}^{3};\mathbb{R}^{3})$ such that $\left\Vert 
\mathbf{v}\right\Vert _{0}:=\sup_{x\in \mathbb{R}^{3}}\left\vert \mathbf{v}%
\left( x\right) \right\vert <\infty $. For any $\alpha \in \left( 0,1\right) 
$ we denote by $C_{b}^{\alpha }(\mathbb{R}^{3};\mathbb{R}^{3})$ the space of
all $\mathbf{v}\in C_{b}(\mathbb{R}^{3};\mathbb{R}^{3})$ such that $\left[ 
\mathbf{v}\right] _{\alpha }:=\sup_{x,y\in \mathbb{R}^{3},x\neq y}\frac{%
\left\vert \mathbf{v}\left( x\right) -\mathbf{v}\left( y\right) \right\vert 
}{\left\vert x-y\right\vert ^{\alpha }}<\infty $; the space $C_{b}^{\alpha }(%
\mathbb{R}^{3};\mathbb{R}^{3})$ is endowed with the norm $\left\Vert \mathbf{%
v}\right\Vert _{\alpha }=\left\Vert \mathbf{v}\right\Vert _{0}+\left[ 
\mathbf{v}\right] _{\alpha }$. We denote by $C_{c}^{\infty }(\mathbb{R}^{3};%
\mathbb{R}^{3})$ the space of all $\mathbf{v}\in C^{\infty }(\mathbb{R}^{3};%
\mathbb{R}^{3})$ which have compact support.

For $p\geq 1$, we denote by $L_{loc}^{p}(\mathbb{R}^{3};\mathbb{R}^{3})$ the
space of measurable vector fields $\mathbf{v}:\mathbb{R}^{3}\rightarrow 
\mathbb{R}^{3}$ such that $\int_{\left\vert x\right\vert \leq R}\left\vert 
\mathbf{v}\left( x\right) \right\vert ^{p}dx<\infty $ for all $R>0$; we
write $\mathbf{v}\in L^{p}(\mathbb{R}^{3};\mathbb{R}^{3})$ when $\int_{%
\mathbb{R}^{3}}\left\vert \mathbf{v}\left( x\right) \right\vert
^{p}dx<\infty $. The notation $\left\langle {\mathbf{v}},{\mathbf{w}}%
\right\rangle $ stands for $\int_{\mathbb{R}^{3}}{\mathbf{v}}\left( x\right)
\cdot {\mathbf{w}}\left( x\right) dx$, when ${\mathbf{v}},{\mathbf{w\in }}%
L^{2}(\mathbb{R}^{3};\mathbb{R}^{3})$.

If ${\mathbf{v}}:\left[ 0,T\right] \times \mathbb{R}^{3}\rightarrow \mathbb{R%
}^{3}$, we usually write ${\mathbf{v}}\left( t,x\right) $, but also $\mathbf{%
v}_{t}$ to denote the function $x\mapsto {\mathbf{v}}\left( t,x\right) $ at
given $t\in \left[ 0,T\right] $.

If ${\mathbf{v}}\in \mathbb{R}^{3}$ we write ${\mathbf{v}}\cdot \nabla $ for
the differential operator $\sum_{i=1}^{3}v^{i}\partial _{x_{i}}$. If ${%
\mathbf{v}},\mathbf{B}:\mathbb{R}^{3}\rightarrow \mathbb{R}^{3}$ the
notation $\left( {\mathbf{v}}\cdot \nabla \right) \mathbf{B}$ stands for the
vector field with components $\left( {\mathbf{v}}\cdot \nabla \right) B^{i}$%
. Similarly, we interpret componentwise operations like $\partial _{k}%
\mathbf{B}$, $\Delta \mathbf{B}$.

\section{Example of blow-up in the deterministic case\label{section 2}}

In this section we consider equation (\ref{vectadv}) in the deterministic
case $\sigma =0$. We give an example of H\"{o}lder continuous bounded vector
field {$\mathbf{v}$} such that $\sup_{\left\vert x\right\vert \leq
1}\left\vert \mathbf{B}\left( t,x\right) \right\vert =+\infty $ for all $t>0$%
, although $\sup_{x\in \mathbb{R}^{3}}\left\vert \mathbf{B}_{0}\left(
x\right) \right\vert <\infty $ and $\mathbf{B}_{0}$\ is smooth.

Let us also remark that, on the contrary, when {$\mathbf{v}$} is of class $%
C\left( \left[ 0,T\right] ;C_{b}^{1}(\mathbb{R}^{3};\mathbb{R}^{3})\right) $%
, for every $\mathbf{B}_{0}\in C(\mathbb{R}^{3};\mathbb{R}^{3})$ there
exists a unique continuous weak solution $\mathbf{B}$ (the definition is
analogous to Definition \ref{def distrib sol} below and the proof is similar
to the one of Theorem \ref{noblowup}); it satisfies identity (\ref{reprform}%
) below where $\Phi _{t}(x)$ is the deterministic flow given by the equation
of characteristics%
\begin{equation*}
\frac{d}{dt}\Phi _{t}(x)={\mathbf{v}}(t,\Phi _{t}(x)),\qquad \Phi _{0}(x)=x.
\end{equation*}%
When {$\mathbf{v}$} is of class $C\left( \left[ 0,T\right] ;C_{b}^{2}(%
\mathbb{R}^{3};\mathbb{R}^{3})\right) $ and $\mathbf{B}_{0}\in C^{1}(\mathbb{%
R}^{3};\mathbb{R}^{3})$ the solution $\mathbf{B}$ is of class $C\left( \left[
0,T\right] ;C^{1}(\mathbb{R}^{3};\mathbb{R}^{3})\right) $, and so on, from
identity (\ref{reprform}). The idea of the example of blow-up comes from
identity (\ref{reprform}): one has to construct a flow $\Phi _{t}(x)$,
corresponding to a vector field {$\mathbf{v}$} less regular than $C\left( %
\left[ 0,T\right] ;C_{b}^{1}(\mathbb{R}^{3};\mathbb{R}^{3})\right) $, such
that $D\Phi _{t}(x)$ blows-up at some point.

\subsection{Preliminaries on cylindrical coordinates}

Limited to this and next subsection, we denote points of $\mathbb{R}^{3}$ by 
$\left( x,y,z\right) $ instead of $x$ (and analogous notations for Euclidea
coordinates). Let us recall that the material derivative, in cylindrical
coordinates, for vectors $\mathbf{A}=\mathbf{A}(r,\theta ,z)$, $\mathbf{B}=%
\mathbf{B}(r,\theta ,z)$, $\mathbf{A}=A_{r}e_{r}+A_{\theta }e_{\theta
}+A_{z}e_{z}$, $\mathbf{B}=B_{r}e_{r}+B_{\theta }e_{\theta }+B_{z}e_{z}$
(where $e_{r}=\frac{x}{r}e_{x}+\frac{y}{r}e_{y}$, $e_{\theta }=-\frac{y}{r}%
e_{x}+\frac{x}{r}e_{y}$) are given by the formula 
\begin{align}
\left( {\mathbf{A}}\cdot \nabla \right) \mathbf{B}& =(A_{r}\frac{\partial
B_{r}}{\partial r}+\frac{A_{\theta }}{r}\frac{\partial B_{r}}{\partial
\theta }+A_{z}\frac{\partial B_{r}}{\partial z}-\frac{A_{\theta }B_{\theta }%
}{r})e_{r}  \notag \\
& +(A_{r}\frac{\partial B_{\theta }}{\partial r}+\frac{A_{\theta }}{r}\frac{%
\partial B_{\theta }}{\partial \theta }+A_{z}\frac{\partial B_{\theta }}{%
\partial z}+\frac{A_{\theta }B_{r}}{r})e_{\theta }  \notag \\
& +(A_{r}\frac{\partial B_{z}}{\partial r}+\frac{A_{\theta }}{r}\frac{%
\partial B_{z}}{\partial \theta }+A_{z}\frac{\partial B_{z}}{\partial z}%
)e_{z}  \notag
\end{align}%
Consequently, 
\begin{align}
\left( {\mathbf{A}}\cdot \nabla \right) \mathbf{B}& -\left( {\mathbf{B}}%
\cdot \nabla \right) \mathbf{A}=(A_{r}\frac{\partial B_{r}}{\partial r}-B_{r}%
\frac{\partial A_{r}}{\partial r}+\frac{A_{\theta }}{r}\frac{\partial B_{r}}{%
\partial \theta }-\frac{B_{\theta }}{r}\frac{\partial A_{r}}{\partial \theta 
}+A_{z}\frac{\partial B_{r}}{\partial z}-B_{z}\frac{\partial A_{r}}{\partial
z})e_{r}  \notag \\
& +(A_{r}\frac{\partial B_{\theta }}{\partial r}-B_{r}\frac{\partial
A_{\theta }}{\partial r}+\frac{A_{\theta }}{r}\frac{\partial B_{\theta }}{%
\partial \theta }-\frac{B_{\theta }}{r}\frac{\partial A_{\theta }}{\partial
\theta }+A_{z}\frac{\partial B_{\theta }}{\partial z}-B_{z}\frac{\partial
A_{\theta }}{\partial z}+\frac{A_{\theta }B_{r}-B_{\theta }A_{r}}{r}%
)e_{\theta }  \notag \\
& +(A_{r}\frac{\partial B_{z}}{\partial r}-B_{r}\frac{\partial A_{z}}{%
\partial r}+\frac{A_{\theta }}{r}\frac{\partial B_{z}}{\partial \theta }-%
\frac{B_{\theta }}{r}\frac{\partial A_{z}}{\partial \theta }+A_{z}\frac{%
\partial B_{z}}{\partial z}-B_{z}\frac{\partial A_{z}}{\partial z})e_{z} 
\notag
\end{align}

With these preliminaries, let us consider a vector field $\mathbf{v}$ of the
form 
\begin{equation*}
\mathbf{v}=v_{\theta }e_{\theta },\quad v_{\theta }=v_{\theta }(r)
\end{equation*}%
and assume that $\mathbf{B}(t)=B_{r}(t)e_{r}+B_{\theta }(t)e_{\theta
}+B_{z}(t)e_{z},t\geq 0$ is a vector field of class $C^{1}$ on $\mathbb{R}%
^{3}\backslash \left\{ 0\right\} $ which satisfies (on $\mathbb{R}%
^{3}\backslash \left\{ 0\right\} $) the equation 
\begin{equation*}
\frac{\partial \mathbf{B}}{\partial t}+\func{curl}(\mathbf{v}\times \mathbf{B%
})=0
\end{equation*}%
with divergence-free initial condition $\mathbf{B}_{0}$. Notice that $\func{%
div}${$\mathbf{v}$}$=\func{div}${$\mathbf{B}$}$=0$. Indeed, $\mathbf{v}$ is
divergence free vector field by definition and $\frac{\partial \func{div}{%
\mathbf{B}}}{\partial t}=-\func{div}\func{curl}(\mathbf{v}\times \mathbf{B}%
)=0$. Hence we can rewrite equation for $\mathbf{B}$ as follows 
\begin{equation*}
\frac{\partial \mathbf{B}}{\partial t}+\left( {\mathbf{v}}\cdot \nabla
\right) \mathbf{B}-\left( \mathbf{B}\cdot \nabla \right) \mathbf{v}=0.
\end{equation*}%
Consequently, in cylindrical coordinates we have 
\begin{equation*}
\frac{\partial B_{r}}{\partial t}+\frac{v_{\theta }}{r}\frac{\partial B_{r}}{%
\partial \theta }=0,
\end{equation*}%
\begin{equation*}
\frac{\partial B_{\theta }}{\partial t}=B_{r}\frac{\partial v_{\theta }}{%
\partial r}-\frac{v_{\theta }}{r}\frac{\partial B_{\theta }}{\partial \theta 
}-\frac{v_{\theta }}{r}B_{r}=-\frac{v_{\theta }}{r}\frac{\partial B_{\theta }%
}{\partial \theta }+B_{r}(\frac{\partial v_{\theta }}{\partial r}-\frac{%
v_{\theta }}{r}),
\end{equation*}%
\begin{equation*}
\frac{\partial B_{z}}{\partial t}+\frac{v_{\theta }}{r}\frac{\partial B_{z}}{%
\partial \theta }=0.
\end{equation*}

\subsection{The example}

Choose, for some $\alpha \in (0,1)$, 
\begin{equation*}
v_{\theta }(r)=r^{\alpha },\qquad \text{\ for }r\in \left[ 0,1\right]
\end{equation*}%
and define $v_{\theta }$ for $r>1$ in a such way that $v_{\theta }\in
C^{\infty }$, $v_{\theta }>0$ and $v_{\theta }(r)\leq e^{-\gamma r}$ , $%
\gamma >0$, $r\geq A>1$ for some $\gamma ,A$.

Then we have, for $r\in (0,1)$, 
\begin{align}
\frac{\partial B_{r}}{\partial t}+r^{\alpha -1}\frac{\partial B_{r}}{%
\partial \theta }& =0,  \notag \\
\frac{\partial B_{\theta }}{\partial t}+r^{\alpha -1}\frac{\partial
B_{\theta }}{\partial \theta }+(1-\alpha )B_{r}r^{\alpha -1}& =0  \notag \\
\frac{\partial B_{z}}{\partial t}+r^{\alpha -1}\frac{\partial B_{z}}{%
\partial \theta }& =0.  \notag
\end{align}%
Hence we can deduce that (we write $B_{r}^{0},B_{z}^{0},B_{\theta }^{0}$ for
the coordinates of $\mathbf{B}_{0}$), for $r\in (0,1)$, 
\begin{equation*}
B_{r}(t,r,\theta ,z)=B_{r}^{0}(r,\theta -r^{\alpha -1}t,z)
\end{equation*}%
\begin{equation*}
B_{z}(t,r,\theta ,z)=B_{z}^{0}(r,\theta -r^{\alpha -1}t,z)
\end{equation*}%
\begin{equation*}
B_{\theta }(t,r,\theta ,z)=B_{\theta }^{0}(r,\theta -r^{\alpha
-1}t,z)-(1-\alpha )\fbox{r$^{\alpha -1}$}tB_{r}^{0}(r,\theta -r^{\alpha
-1}t,z).
\end{equation*}%
A non-zero radial component $B_{r}^{0}$ of the initial condition near the
vertical axis for the origin ($r=0$) yields a blow-up of the angular
component $B_{\theta }$.

Thus we see that any smooth bounded initial condition $\mathbf{B}_{0}$, such
that $B_{r}^{0}(r,\theta ,z)>0$ for all values of the arguments, gives rise
to a solution $\mathbf{B}$\ such that 
\begin{equation*}
\lim_{r\rightarrow 0}\left\vert B_{\theta }(t,\theta ,r,z)\right\vert =\infty
\end{equation*}%
for any $t>0$, at any point $\left( \theta ,z\right) $. From this one
deduces $\lim_{\left( x,y,z\right) \rightarrow 0}\left\vert \mathbf{B}\left(
t,x,y,z\right) \right\vert =+\infty $ for all $t>0$ (since $B_{\theta }$ is
the projection on $e_{\theta }$ at $\left( x,y,z\right) $; similarly for $%
B_{r},B_{z}$;\ thus the divergence of $B_{\theta }$ and boundedness of $%
B_{r},B_{z}$ imply the divergence of $\mathbf{B}$).

\begin{remark}
With more work, taking a time-dependent vector field $\mathbf{v}$\ which is
smooth until time $t_{0}>0$ when it develops an H\"{o}lder singularity of
the form above, one can construct an example of solution $\mathbf{B}$\ which
is smooth on $[0,t_{0})$ but infinite at some point at time $t_{0}$. Such
example would mimic more closely what maybe could happen in a non-passive
version of the vector advection equation.
\end{remark}

\subsection{The Lagrangian picture}

We summarize here the features of this example, with the following items and
some pictures (just to give a graphical intuition of what happens).

i) The fluid rotates around the vertical $z$-axis $\zeta $ at the origin;
the Lagrangian particles describe circles around $\zeta $, the Cauchy
problem 
\begin{equation}
\frac{d}{dt}X_{t}={\mathbf{v}}(t,X_{t}),\qquad X_{0}=x  \label{ODE}
\end{equation}%
is uniquely solvable and generate a continuous flow $\Phi _{t}(x)$. Figure 1
shows a number of Lagrangian trajectories (solutions of the Cauchy problem (%
\ref{ODE}))\ starting on the $x$-axis, in the regular case $\alpha =1$,
where the velocity produces a rigid motion (no singularity). Figure 2 shows
the case $\alpha =0.2$, where the velocity of rotation near the origin is so
large (still infinitesimal, so that the velocity field is H\"{o}lder
continuous) that very close initial particles are displaced a lot;\ and the
ratio between the displacement at time $t$ and that at time zero diverges
when the particles approach zero.

\begin{multicols}{2}
\includegraphics[width=7cm]{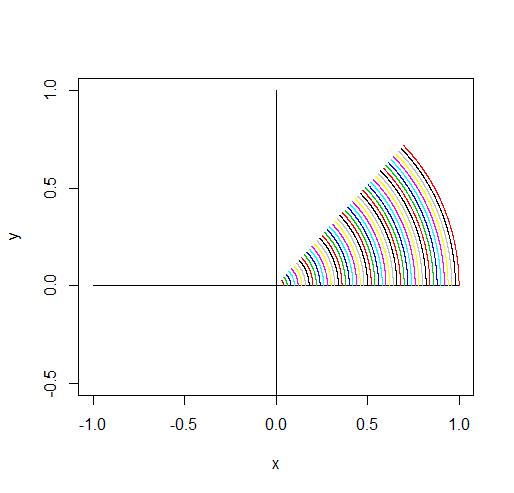}\\
Fig.1. Lagrangian trajectories for $\protect\alpha =1$.\columnbreak

\includegraphics[width=7cm]{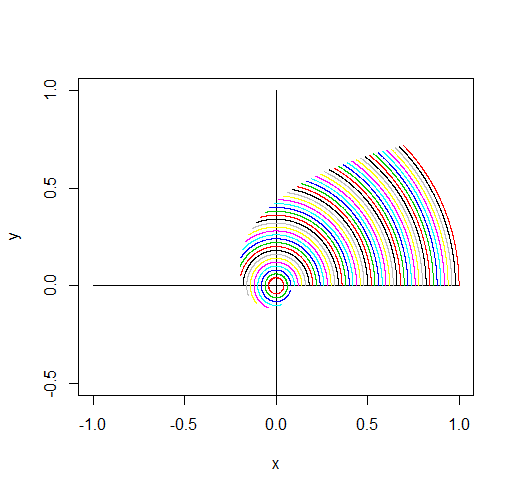}\\
Fig.2. Lagrangian trajectories for $\protect\alpha =0.2$.
\end{multicols}


ii) The flow $\Phi _{t}(x)$ is however not differentiable at the vertical
axis $\zeta $ (it is smooth outside $\zeta $), as it may be guesses from
Figure 2; ideal lines of Lagrangian points in a plane orthogonal to $\zeta $
are stretched near $\zeta $ and the stretching becomes infinite at $\zeta $,
see Figure 3 below.

iii) The passive field $\mathbf{B}$\ is also stretched by the fluid and the
stretching blows-up at $\zeta $.

With this picture in mind, we may anticipate the behavior when we add noise.
As we shall see below, the transport type noise, in Stratonovich form,
introduced in equation (\ref{vectadv}), corresponds at the Lagrangian level
to the addition of a random shift to all Lagrangian particles (see equation (%
\ref{SDE})). Figures 3 and 4 below show the time evolution of the ideal line
initially equal to the $x$ axis. In the deterministic case (Figure 3) this
line is infinitely stretched near the origin. In the stochastic case (Figure
4), even with very small noise intensity ($\sigma =0.1)$, the line is
shifted by noise a little bit in all possible directions and thus it passes
through the origin only for a negligible amount of time. Stretching still
occurs but not with infinite strength and the visible result is that the
line at the forth time instant looks still smooth although strongly curved.

Stretching still exists but it is smeared-out, distributed among different
portions of fluid; the deterministic concentration of stretching at $\zeta $
is broken.

\begin{multicols}{2}
\includegraphics[width=7cm]{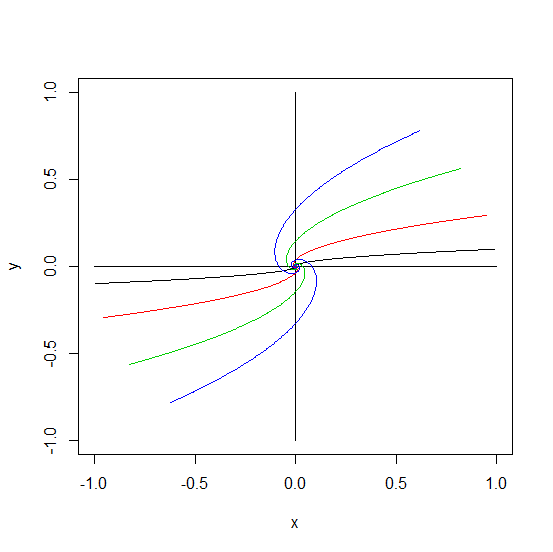}\\
Fig.3. Ideal lines evolution, no noise.\columnbreak

\includegraphics[width=7cm]{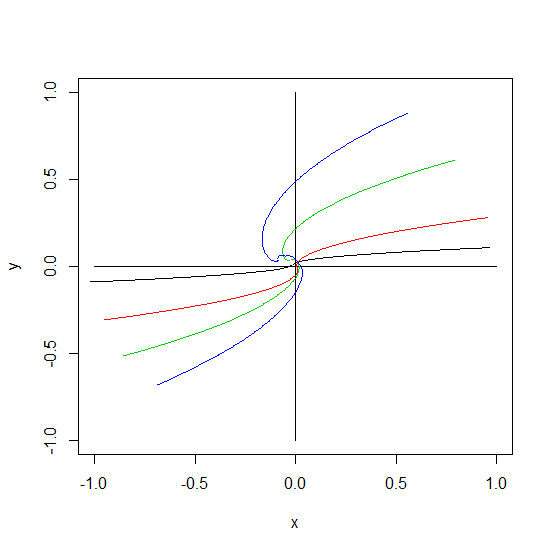}\\
Fig.4. Ideal lines evolution, noise with $\protect\sigma =0.1$.
\end{multicols}


\section{The stochastic case: absence of blow-up}

\subsection{The regular case \label{section 3}}

In this section we study the regular case. Let $\mathbf{W}=\left(
W^{1},W^{2},W^{3}\right) $ be a $3$-dimensional Brownian motion on a
probability space $(\Omega ,\mathcal{A},P)$ and let $(\mathcal{F}_{t})_{t}$
be its natural completed filtration. Let {$\mathbf{v}$} be a divergence-free
vector field in $C^{1}([0,T];C_{c}^{\infty }(\mathbb{R}^{3};\mathbb{R}^{3}))$
and $\mathbf{B}_{0}$ be a divergence-free vector field in $C_{c}^{\infty }(%
\mathbb{R}^{3};\mathbb{R}^{3})$. For a divergence-free solution $\mathbf{B}$%
, equation (\ref{vectadv}) reads formally 
\begin{equation}
d\mathbf{B}+\left( \left( {\mathbf{v}}\cdot \nabla \right) \mathbf{B}-\left( 
\mathbf{B}\cdot \nabla \right) \mathbf{v}\right) dt+\sigma \sum_{k}\partial
_{k}\mathbf{B}\circ dW^{k}=0.  \label{vectadv2}
\end{equation}%
We will always use (\ref{vectadv}) in this form.

\begin{remark}
The Stratonovich operation $\partial _{k}\mathbf{B}\circ dW^{k}$ is the
natural one from the physical viewpoint, because of Wong-Zakai principle,
see the Appendix of \cite{FlaGubPri} for an example, and because of the
formal validity of conservation laws. More rigorously, it is responsible for
the validity of relation (\ref{reprform}) between $\mathbf{B}$ and the
Lagrangian motion, relation which extends to the stochastic case a well know
deterministic relation.
\end{remark}

For mathematical convenience, we translate Stratonovich in It\^{o} form.
Formally, the martingale part of $\partial _{k}\mathbf{B}$ is (from equation
(\ref{vectadv2}) itself) equal to $\sigma \sum_{j}\partial _{k}\partial _{j}%
\mathbf{B}dW^{j}$ and thus the quadratic variation $d\left[ \partial _{k}%
\mathbf{B,}W^{k}\right] $ is equal to $\sigma \partial _{k}\partial _{k}%
\mathbf{B}$; therefore 
\begin{equation*}
\sigma \sum_{k}\partial _{k}\mathbf{B}\circ dW^{k}=\sigma \sum_{k}\partial
_{k}\mathbf{B}dW^{k}+\frac{\sigma ^{2}}{2}\Delta \mathbf{B.}
\end{equation*}%
This is the heuristic justification of the following rigorous definition.

\begin{definition}
A regular solution to (\ref{vectadv}) is a vector field $\mathbf{B}%
:[0,T]\times \mathbb{R}^{3}\times \Omega \rightarrow \mathbb{R}^{3}$ such
that

i) $\mathbf{B}(t,x)$ and its derivatives in $x$ up to fourth order are
continuous in $(t,x)$

ii)\ for every $i,j=1,...,d$ and $x\in \mathbb{R}^{3}$, $\mathbf{B}(t,x)$, $%
\partial _{x_{i}}\mathbf{B}(t,x)$, $\partial _{x_{j}}\partial _{x_{i}}%
\mathbf{B}(t,x)$ are adapted processes

ii) for every $(t,x)$, $\mathrm{div}\mathbf{B}(t,x)=0$ and%
\begin{align*}
\mathbf{B}(t,x)& =\mathbf{B}_{0}(x)+\int_{0}^{t}\left[ \left( \mathbf{B}%
(r,x)\cdot \nabla \right) \mathbf{v}(r,x)-\left( {\mathbf{v}}(r,x)\cdot
\nabla \right) \mathbf{B}(r,x)\right] dr \\
& -\sigma \sum_{k=1}^{3}\int_{0}^{t}\partial _{k}\mathbf{B}(r,x)dW_{r}^{k}+%
\frac{\sigma ^{2}}{2}\int_{0}^{t}\Delta \mathbf{B}(r,x)dr.
\end{align*}
\end{definition}

\begin{remark}
In order to give a meaning to the equation it is not necessary to ask $C^{4}$
regularity in $x$ in point (i); the requirement is imposed to apply It\^{o}%
-Kunita-Wentzell formula (Theorem 3.3.1 in \cite{Kun}) below.
\end{remark}

\begin{remark}
For the purpose of this paper, one can simplify and ask that $\mathbf{B}$ is 
$C^{\infty }$ in $x$, with all derivatives continuous in $(t,x)$;\ the
results below remain true.
\end{remark}

Consider now the SDE on $\mathbb{R}^{3}$ 
\begin{equation}
dX_{t}={\mathbf{v}}(t,X_{t})dt+\sigma d\mathbf{W}_{t},\qquad X_{0}=x.
\label{SDE}
\end{equation}%
It is a classical result (see \cite{Kun}) that there exists a stochastic
flow $\Phi $ of $C^{\infty }$ diffeomorphisms (see Definition \ref{defdiffeo}
in Section \ref{sect Holder drift}) solving the above SDE. Since {$\mathbf{v}
$} is divergence-free, $\Phi _{t}$ and $\Phi _{t}^{-1}$ are also
measure-preserving for every $t$, i.e.\ $det(D\Phi _{t})=1$.

We can now prove the representation formula for the regular solution to
equation (\ref{vectadv}), which will be the key ingredient of our work.

\begin{proposition}
\label{regcase}Suppose $\mathbf{B}_{0}\in C_{c}^{\infty }(\mathbb{R}^{3};%
\mathbb{R}^{3})$ and {$\mathbf{v}$}$\in C^{1}([0,T];C_{c}^{\infty }(\mathbb{R%
}^{3};\mathbb{R}^{3}))$, both divergence free. Then equation (\ref{vectadv})
admits a unique regular solution, satisfying the identity 
\begin{equation}
\mathbf{B}(t,\Phi _{t}(x))=D\Phi _{t}(x)\mathbf{B}_{0}(x).  \label{reprform}
\end{equation}
\end{proposition}

\begin{remark}
\label{otherrepr}Notice that $D\Phi _{t}(\Phi _{t}^{-1}(x))=(D\Phi
_{t}^{-1}(x))^{-1}$. This inverse matrix is the transpose of the cofactor
matrix of $D\Phi _{t}^{-1}$, multiplied by the inverse of the determinant of 
$D\Phi _{t}^{-1}(x)$, which is $1$ since $\Phi _{t}$ is measure-preserving;
the cofactor matrix of a given $3\times 3$ matrix $A$ is a polynomial
function $H(A)^{T}$, of degree $2$, of $A$. So we have $D\Phi _{t}(\Phi
_{t}^{-1}(x))=H(D\Phi _{t}^{-1}(x))$ and formula (\ref{reprform}) also reads 
\begin{equation}
\mathbf{B}(t,x)=H(D\Phi _{t}^{-1}(x))\mathbf{B}_{0}(\Phi _{t}^{-1}(x)).
\label{reprform 2}
\end{equation}
\end{remark}

\begin{proof}
\textbf{Step 1 }(chain rule). Let us recall the so called It\^{o}%
-Kunita-Wentzell formula (Theorem 8.1 in \cite{KunSF}; see also Theorem
3.3.1 in \cite{Kun} for a variant). We state it with the notations of
interest for us. Assume that $F\left( t,x\right) $, $t\in \left[ 0,T\right] $%
, $x\in \mathbb{R}^{d}$, is a continuous random field, twice differentiable
in $x$ with second derivatives continuous in $\left( t,x\right) $, of the
form%
\begin{equation*}
F\left( t,x\right) =F_{0}\left( x\right) +\int_{0}^{t}f_{0}\left( s,x\right)
ds+\sum_{k=1}^{n}\int_{0}^{t}f_{k}\left( s,x\right) dW_{s}^{k}
\end{equation*}%
where $W^{k}$, $k=1,...,n$ are independent Brownian motions and $f_{k}$, $%
k=0,1,...,n$ are twice differentiable in $x$, continuous in $\left(
t,x\right) $ with their second space derivatives, and for each $x$ the
processes $t\mapsto f_{k}\left( t,x\right) $ are adapted. Let $X_{t}$ be a
continuous semimartingale in $\mathbb{R}^{d}$. Then%
\begin{eqnarray*}
F\left( t,X_{t}\right) &=&F_{0}\left( X_{0}\right) +\int_{0}^{t}f_{0}\left(
s,X_{s}\right) ds+\sum_{k=1}^{n}\int_{0}^{t}f_{k}\left( s,X_{s}\right)
dW_{s}^{k} \\
&&+\int_{0}^{t}\nabla F\left( t,X_{s}\right) \cdot dX_{s}+\frac{1}{2}%
\sum_{k,h=1}^{n}\int_{0}^{t}\partial _{x_{k}}\partial _{x_{h}}F\left(
t,X_{s}\right) d\left[ X^{h},X^{k}\right] _{s} \\
&&+\sum_{k=1}^{n}\sum_{i=1}^{d}\int_{0}^{t}\partial _{x_{i}}f_{k}\left(
s,X_{s}\right) d\left[ X^{i},W^{k}\right] _{s}
\end{eqnarray*}%
where $\left[ X^{h},X^{k}\right] _{t}$ denotes the quadratic mutual
variation between the components of $X$ and similarly for $\left[ X^{h},W^{k}%
\right] _{t}$.

\textbf{Step 2 }(uniqueness). Fix $x$ in $\mathbb{R}^{3}$; observe that $%
D\Phi _{t}(x)\mathbf{B}_{0}(x)$ is the unique solution to 
\begin{equation}
\frac{d\mathbf{Z}_{t}}{dt}=\left( \mathbf{Z}_{t}\cdot \nabla \right) \mathbf{%
v}(t,\Phi _{t}(x)).  \label{derivSDE}
\end{equation}%
with $\mathbf{Z}_{0}=\mathbf{B}_{0}(x)$ (uniqueness follows from the fact
the the stochastic drift for this ODE, namely $(t,y)\rightarrow D${$\mathbf{v%
}$}$(t,\Phi _{t}(x))y$ is in $C_{b}^{1}$). Thus, in order to get uniqueness
for equation (\ref{vectadv}) and prove formula (\ref{reprform}), it is
enough to prove that, for any regular solution $\mathbf{B}$ to (\ref{vectadv}%
), $\mathbf{B}(t,\Phi _{t}(x))$ satisfies equation (\ref{derivSDE}). For
this purpose, we use the chain rule of Step 1 (the assumptions in the
definition of regular solution\ above are imposed precisely in order to
apply this result). For each component $j=1,2,3$ we apply the formula with%
\begin{equation*}
F=B^{j},f_{0}=\left( \left( \mathbf{B}\cdot \nabla \right) \mathbf{v}-\left( 
{\mathbf{v}}\cdot \nabla \right) \mathbf{B}\right) ^{j}+\frac{\sigma ^{2}}{2}%
\Delta B^{j},f_{k}=-\sigma \partial _{k}B^{j},X_{t}=\Phi _{t}(x).
\end{equation*}%
The result, rewritten in vector form, is%
\begin{eqnarray*}
d[\mathbf{B}(t,\Phi _{t}(x))] &=&\left( d\mathbf{B}\right) (t,\Phi _{t}(x))
\\
&&+\sum_{i=1}^{3}\partial _{x_{i}}\mathbf{B}(t,\Phi _{t}(x))d\Phi _{t}^{i}(x)
\\
&&+\frac{\sigma ^{2}}{2}\Delta \mathbf{B}(t,\Phi _{t}(x))dt-\sigma
^{2}\Delta \mathbf{B}(t,\Phi _{t}(x))dt
\end{eqnarray*}%
because%
\begin{equation*}
\sum_{k=1}^{3}\sum_{i=1}^{3}\int_{0}^{t}\partial _{x_{i}}f_{k}d\left[
X^{i},W^{k}\right] _{s}=-\sigma
\sum_{k=1}^{3}\sum_{i=1}^{3}\int_{0}^{t}\partial _{x_{i}}\partial
_{k}B^{j}\sigma \delta _{ik}ds=-\sigma ^{2}\int_{0}^{t}\Delta B^{j}ds
\end{equation*}%
(since $d\left[ X^{i},W^{k}\right] _{s}=\sigma \delta _{ik}ds$). Therefore%
\begin{eqnarray*}
d[\mathbf{B}(t,\Phi _{t}(x))] &=&\left( \left( \mathbf{B}\cdot \nabla
\right) \mathbf{v}-\left( {\mathbf{v}}\cdot \nabla \right) \mathbf{B}\right)
dt+\frac{\sigma ^{2}}{2}\Delta \mathbf{B}dt-\sigma \sum_{k=1}^{3}\partial
_{k}\mathbf{B}dW_{t}^{k} \\
&&+\left( {\mathbf{v}}\cdot \nabla \right) \mathbf{B}dt+\sigma
\sum_{k=1}^{3}\partial _{k}\mathbf{B}dW_{t}^{k} \\
&&-\frac{\sigma ^{2}}{2}\Delta \mathbf{B}(t,\Phi _{t}(x))dt
\end{eqnarray*}%
\begin{equation*}
=\left( \left( \mathbf{B}\cdot \nabla \right) \mathbf{v}\right) (t,\Phi
_{t}(x))dt.
\end{equation*}%
Therefore $\mathbf{B}(t,\Phi _{t}(x))$ satisfies (\ref{derivSDE}).
Uniqueness and formula (\ref{reprform}) are proved.

\textbf{Step 3 }(existence). Conversely, given $\mathbf{B}$ defined by (\ref%
{reprform}), let us prove that it is a regular solution to equation (\ref%
{vectadv}). Properties (i)-(ii)\ of the definition of regular solution are
obvious from (\ref{reprform 2}) (indeed $\mathbf{B}$ is $C^{\infty }$ in $x$%
). It also follows, from It\^{o}-Kunita-Wentzell formula,\ that $\mathbf{B}%
\left( t,x\right) $ has the form 
\begin{equation}
d\mathbf{B}\left( t,x\right) =\mathbf{A}\left( t,x\right) dt+\sum_{k=1}^{3}%
\mathbf{S}_{k}\left( t,x\right) dW_{t}^{k}  \label{decomp}
\end{equation}%
where $\mathbf{B}\left( t,x\right) ,\mathbf{A}\left( t,x\right) ,\mathbf{S}%
_{k}\left( t,x\right) $ are continuous in $\left( t,x\right) $ with their
second space derivatives, and are adapted in $t$ for every $x$. Notice that
we do not need to compute explicitly $\mathbf{A}\left( t,x\right) $ and $%
\mathbf{S}_{k}\left( t,x\right) $ (by It\^{o}-Kunita-Wentzell formula) from
the identity (\ref{reprform 2}) (this would involve too complex expressions
with derivatives of the flow). We just need to realize that It\^{o}%
-Kunita-Wentzell formula can be applied and gives a decomposition of the
form (\ref{decomp}) with $\mathbf{B}\left( t,x\right) ,\mathbf{A}\left(
t,x\right) ,\mathbf{S}_{k}\left( t,x\right) $ having the regularity stated
above.

Thanks to this regularity, we may apply It\^{o}-Kunita-Wentzell formula to $%
\mathbf{B}(t,\Phi _{t}(x))$, where now we only know that identities (\ref%
{decomp}) and (\ref{reprform}) are satisfied by $\mathbf{B}$. On one side,
from (\ref{reprform}) and the fact that $D\Phi _{t}(x)\mathbf{B}_{0}(x)$ is
the unique solution to (\ref{derivSDE}) we get%
\begin{equation*}
d[\mathbf{B}(t,\Phi _{t}(x))]=\left( \mathbf{B}(t,\Phi _{t}(x))\cdot \nabla
\right) \mathbf{v}(t,\Phi _{t}(x))dt.
\end{equation*}%
On the other side, similarly to the computation of Step 2, from It\^{o}%
-Kunita-Wentzell formula applied to the function%
\begin{equation*}
F=B^{j},f_{0}=A^{j},f_{k}=S_{k}^{j},X_{t}=\Phi _{t}(x)
\end{equation*}%
we get 
\begin{eqnarray*}
d[\mathbf{B}(t,\Phi _{t}(x))] &=&\left( d\mathbf{B}\right) (t,\Phi _{t}(x))
\\
&&+\sum_{i=1}^{3}\partial _{x_{i}}\mathbf{B}(t,\Phi _{t}(x))d\Phi _{t}^{i}(x)
\\
&&+\frac{\sigma ^{2}}{2}\Delta \mathbf{B}(t,\Phi _{t}(x))dt+\sigma
\sum_{k=1}^{3}\partial _{x_{k}}\mathbf{S}_{k}(t,\Phi _{t}(x))dt
\end{eqnarray*}%
because now%
\begin{equation*}
\sum_{k=1}^{3}\sum_{i=1}^{3}\int_{0}^{t}\partial _{x_{i}}f_{k}d\left[
X^{i},W^{k}\right] _{s}=\sum_{k=1}^{3}\sum_{i=1}^{3}\int_{0}^{t}\partial
_{x_{i}}S_{k}^{j}\sigma \delta _{ik}ds=\sigma
\sum_{k=1}^{3}\int_{0}^{t}\partial _{x_{k}}S_{k}^{j}ds.
\end{equation*}%
Therefore%
\begin{eqnarray*}
d[\mathbf{B}(t,\Phi _{t}(x))] &=&{\mathbf{A}}(t,\Phi
_{t}(x))dt+\sum_{k=1}^{3}\mathbf{S}_{k}(t,\Phi _{t}(x))dW_{t}^{k} \\
&&+\left( \left( {\mathbf{v}}\cdot \nabla \right) \mathbf{B}\right) (t,\Phi
_{t}(x))dt+\sigma \sum_{k=1}^{3}\partial _{k}\mathbf{B}(t,\Phi
_{t}(x))dW_{t}^{k} \\
&&+\frac{\sigma ^{2}}{2}\Delta \mathbf{B}(t,\Phi _{t}(x))dt+\sigma
\sum_{k=1}^{3}\partial _{x_{k}}\mathbf{S}_{k}(t,\Phi _{t}(x))dt.
\end{eqnarray*}%
Equating the two identities satisfied by $d[\mathbf{B}(t,\Phi _{t}(x))]$,
and using the invertibility of $\Phi _{t}$, we get%
\begin{eqnarray*}
\mathbf{S}_{k} &=&-\sigma \partial _{k}\mathbf{B} \\
\left( {\mathbf{B}}\cdot \nabla \right) \mathbf{v} &=&{\mathbf{A}}+\left( {%
\mathbf{v}}\cdot \nabla \right) \mathbf{B}+\frac{\sigma ^{2}}{2}\Delta 
\mathbf{B}+\sigma \sum_{k=1}^{3}\partial _{x_{k}}\mathbf{S}_{k}.
\end{eqnarray*}%
Thus%
\begin{equation*}
{\mathbf{A}}=\left( {\mathbf{B}}\cdot \nabla \right) \mathbf{v-}\left( {%
\mathbf{v}}\cdot \nabla \right) \mathbf{B}+\frac{\sigma ^{2}}{2}\Delta 
\mathbf{B}
\end{equation*}%
which completes the proof that $\mathbf{B}$\ satisfies the SPDE.

It remains to prove the divergence-free property. For this, since $\mathbf{B}
$ is regular, it is enough to show that, for every fixed $t$, for a.e.\ $%
\omega $, $\mathrm{div}\mathbf{B}(t,\cdot ,\omega )$ is $0$ in the sense of
distributions. For this, take $\varphi $ in $C_{c}^{\infty }(\mathbb{R}^{3})$%
; then, using integration by parts (notice that also $\mathbf{B}(\omega )$
has compact support and remember that $\Phi _{t}$ is measure-preserving)%
\begin{equation*}
\int_{\mathbb{R}^{3}}\mathbf{B}_{t}\cdot \nabla \varphi dx=\int_{\mathbb{R}%
^{3}}D\Phi _{t}\mathbf{B}_{0}\cdot \nabla \varphi (\Phi _{t})dx
\end{equation*}%
\begin{equation*}
=\int_{\mathbb{R}^{3}}\mathbf{B}_{0}\cdot (D\Phi _{t})^{T}\nabla \varphi
(\Phi _{t})dx=\int_{\mathbb{R}^{3}}\mathbf{B}_{0}\cdot \nabla \lbrack
\varphi (\Phi _{t})]dx=0
\end{equation*}%
since $\mathbf{B}_{0}$ is divergence-free. The proof is complete.
\end{proof}

\subsection{The case when $\mathbf{v}$ is only H\"{o}lder continuous and
bounded\label{sect Holder drift}}

In this section we shall always assume $\sigma \neq 0$ and the following
condition.

\begin{condition}
\label{Holdercond} The vector field $\mathbf{v}$ is in $C([0,T];C_{b}^{%
\alpha }(\mathbb{R}^{3};\mathbb{R}^{3}))$ for some $\alpha \in \left(
0,1\right) $ and it is divergence-free.
\end{condition}

\begin{definition}
\label{def distrib sol}Let $\mathbf{B}_{0}$ be divergence-free and in $C(%
\mathbb{R}^{3};\mathbb{R}^{3})$. A continuous weak solution to equation (\ref%
{vectadv}) is a vector field $\mathbf{B}:[0,T]\times \mathbb{R}^{3}\times
\Omega \rightarrow \mathbb{R}^{3}$, with a.e. path in $C([0,T]\times \mathbb{%
R}^{3};\mathbb{R}^{3})$, weakly adapted to $(\mathcal{F}_{t})_{t}$ (namely
such that $\langle \mathbf{B},\varphi \rangle $ is adapted for all $\varphi $
in $C_{c}^{\infty }(\mathbb{R}^{3};\mathbb{R}^{3})$) such that:

i) for every $\varphi $ in $C_{c}^{\infty }(\mathbb{R}^{3};\mathbb{R}^{3})$,
the continuous adapted process $\langle \mathbf{B},\varphi \rangle $
satisfies 
\begin{equation}
\langle \mathbf{B}_{t},\varphi \rangle =\langle \mathbf{B}_{0},\varphi
\rangle +\int_{0}^{t}\langle (D\varphi )_{r}^{A}\mathbf{v},\mathbf{B}%
_{r}\rangle dr+\sigma \sum_{k=1}^{d}\int_{0}^{t}\langle (D\varphi )\mathbf{e}%
_{k},\mathbf{B}_{r}\rangle dW_{r}^{k}+\frac{\sigma ^{2}}{2}%
\int_{0}^{t}\langle \Delta \varphi ,\mathbf{B}_{r}\rangle dr,
\label{vectadv_distrib}
\end{equation}%
where $((D\varphi )(x))^{A}=D\varphi (x)-(D\varphi (x))^{T}$ is the
antisymmetric part of the matrix $D\varphi (x)$;

ii) $\mathbf{B}_{t}$ is divergence-free, in the sense that $P\{\mathrm{div}%
\mathbf{B}_{t}=0,\ \forall t\in \left[ 0,T\right] \}=1$.
\end{definition}

Notice that the It\^{o} integrals are well defined since the processes $%
\langle (D\varphi )\mathbf{e}_{k},\mathbf{B}_{r}\rangle $ are continuous and
adapted.

\begin{remark}
One can define a similar notion of $L^{p}$ weak solution and, at least for $%
p>1$, existence and uniqueness should remain true with a more elaborated
proof. We restrict ourselves to continuous solution to emphasize the no
blow-up result.
\end{remark}

The aim of this section is to prove the following main result.

\begin{theorem}
\label{noblowup}Assume that $\sigma \neq 0$. Let $\mathbf{B}_{0}$ be
divergence-free in $C(\mathbb{R}^{3};\mathbb{R}^{3})$ and suppose Condition %
\ref{Holdercond}. Then there exists a unique continuous weak solution $%
\mathbf{B}$ to equation (\ref{vectadv}), starting from $\mathbf{B}_{0}$. In
particular no blow-up occurs.
\end{theorem}

Let us recall the notion of stochastic flow of $C^{1,\beta }$
diffeomorphisms, limited to the properties of interest to us.

\begin{definition}
\label{defdiffeo} A stochastic flow $\Phi $ of $C^{1,\beta }$
diffeomorphisms (on $\mathbb{R}^{3}$), $\beta \in \left( 0,1\right) $, is a
map $[0,T]\times \mathbb{R}^{3}\times \Omega \rightarrow \mathbb{R}^{3}$
such that

\begin{itemize}
\item for every $x$ in $\mathbb{R}^{3}$, $\Phi(x)$ is adapted to $(\mathcal{F%
}_{t})_{t}$;

\item for a.e.\ $\omega$ in $\Omega$, $\Phi(\omega)$ is a flow of $%
C^{1,\beta }$ diffeomorphisms, i.e.\ $\Phi_{0}(\omega)=id$, for every $t$, $%
\Phi _{t}(\omega)$ is a diffeomorphism, $\Phi_{t}$, $\Phi^{-1}_{t}$, $%
D\Phi_{t}$ and $D\Phi^{-1}_{t}$ are jointly continuous on $[0,T]\times%
\mathbb{R}^{3}$, $\beta$-H\"older continuous in space uniformly in time.
\end{itemize}
\end{definition}

In the definition of flows we did not mention the cocycle property, since it
is not useful for our purposes.

We need the following result (valid more in general in $\mathbb{R}^{d}$),
see \cite{FlaGubPri}, Theorem 5.

\begin{theorem}
\label{flowdiffeo}Assume that $\sigma \neq 0$. Let $\mathbf{v}$ satisfy
Condition \ref{Holdercond} and consider the SDE (\ref{SDE}) on $\mathbb{R}%
^{3}$.

\begin{enumerate}
\item For every $x$ in $\mathbb{R}^{3}$, there exists a unique strong
solution to the SDE (\ref{SDE}) starting from $x$. There exists a stochastic
flow of $C^{1,\alpha ^{\prime }}$ diffeomorphisms, for every $\alpha
^{\prime }<\alpha $, solving the SDE and belonging to $L_{loc}^{\infty
}([0,T]\times \mathbb{R}^{3};L^{m}(\Omega ))$ for every finite $m$.

\item Let $(\mathbf{v}^{\epsilon })_{\epsilon >0}$ be a family of
divergence-free vector fields in $C([0,T];C_{b}^{\alpha }(\mathbb{R}^{3};%
\mathbb{R}^{3}))$ converging to $\mathbf{v}$ in this space, as $\epsilon
\rightarrow 0$. For every $\epsilon >0$, let $\Phi ^{\epsilon }$ be the
stochastic flow of diffeomorphisms solving (\ref{SDE}) with drift $\mathbf{v}%
^{\epsilon }$. Then, for every $R>0$ and every $m\geq 1$, the following
results hold: 
\begin{align}
\lim_{\epsilon \rightarrow 0}\sup_{t\in \lbrack 0,T]}\sup_{\left\vert
x\right\vert \leq R}E[|\Phi _{t}^{\epsilon }(x)-\Phi _{t}(x)|^{m}]& =0, \\
\lim_{\epsilon \rightarrow 0}\sup_{t\in \lbrack 0,T]}\sup_{\left\vert
x\right\vert \leq R}E[|D\Phi _{t}^{\epsilon }(x)-D\Phi _{t}(x)|^{m}]& =0
\end{align}%
and the same for the inverse flow $\Phi _{t}^{-1}$ and its derivative in
space.

\item for every $t$, $\Phi_{t}$ is measure-preserving, i.e.\ $det(D\Phi
_{t}(x))=1$ for every $x$ in $\mathbb{R}^{3}$.
\end{enumerate}
\end{theorem}

We split the proof of Theorem \ref{noblowup} in two lemmata, one of
existence and the other of uniqueness.

\begin{lemma}
\label{lemma_reprformula} Let $\mathbf{B}_{0}$ be divergence-free and in $C(%
\mathbb{R}^{3};\mathbb{R}^{3})$ and suppose Condition \ref{Holdercond} hold;
let $\Phi $ be the flow of diffeomorphisms solving the SDE (\ref{SDE}) (as
given in Theorem \ref{flowdiffeo}). Define the random vector field $\mathbf{B%
}$ as 
\begin{equation}
\mathbf{B}(t,x)=D\Phi _{t}(\Phi _{t}^{-1}(x))\mathbf{B}_{0}(\Phi
_{t}^{-1}(x)).  \label{reprform 3}
\end{equation}%
Then $\mathbf{B}$ is a continuous weak solution to equation (\ref{vectadv}).
\end{lemma}

\begin{proof}
\textbf{Step 1 }(regularity). By definition (\ref{reprform 3}), the
assumption on $\mathbf{B}_{0}$\ and the continuity properties in $\left(
t,x\right) $ of $\Phi _{t}^{-1}(x)$ and $D\Phi _{t}(x)$ it follows that $%
\mathbf{B}\in C([0,T]\times \mathbb{R}^{3};\mathbb{R}^{3})$ with probability
one; since $\Phi _{t}^{-1}(x)$ and $D\Phi _{t}(x)$ are $\mathcal{F}_{t}$
measurable, for every $x$ the process $\mathbf{B}(t,x)$ is adapted to $(%
\mathcal{F}_{t})_{t}$, hence also weakly adapted. It remains to prove
properties (i) and (ii)\ of Definition \ref{def distrib sol}.

\textbf{Step 2 }(property (i)). Let $(\mathbf{v}^{\epsilon })_{\epsilon >0}$
be a family of $C^{1}([0,T];C_{c}^{\infty }(\mathbb{R}^{3};\mathbb{R}^{3}))$
divergence-free vector fields, approximating $\mathbf{v}$ in $%
C([0,T];C_{b}^{\alpha }(\mathbb{R}^{3};\mathbb{R}^{3}))$; let $(\mathbf{B}%
_{0}^{\epsilon })_{\epsilon }$ be a family of $C_{c}^{\infty }(\mathbb{R}%
^{3};\mathbb{R}^{3})$ divergence-free vector fields, approximating $\mathbf{B%
}_{0}$ in $C_{b}(\mathbb{R}^{3};\mathbb{R}^{3})$. We know from Lemma \ref%
{regcase} that, for every $\epsilon >0$, 
\begin{equation}
\mathbf{B}^{\epsilon }(t,x)=D\Phi _{t}^{\epsilon }((\Phi _{t}^{\epsilon
})^{-1}(x))\mathbf{B}_{0}^{\epsilon }((\Phi _{t}^{\epsilon })^{-1}(x))
\end{equation}%
solves equation (\ref{vectadv}), where $\Phi ^{\epsilon }$ is the regular
stochastic flow solving the SDE (\ref{SDE}) with drift $\mathbf{v}^{\epsilon
}$. Let us first show that for every $(t,x)$, $(\mathbf{B}^{\epsilon
}(t,x))_{\epsilon }$ converges to $\mathbf{B}(t,x)$, defined by (\ref%
{reprform 3}), in $L^{m}(\Omega ;\mathbb{R}^{3})$, for every finite $m$.

Fix $(t,x)$ and $m\geq 1$. Using Remark \ref{otherrepr} (which also applies
to $\Phi $, since $\det (D\Phi _{t})=1$), We have%
\begin{eqnarray*}
&&|\mathbf{B}^{\epsilon }(t,x)-\mathbf{B}(t,x)| \\
&\leq &|\mathbf{B}_{0}^{\epsilon }((\Phi _{t}^{\epsilon
})^{-1}(x))||H((D\Phi _{t}^{\epsilon })^{-1}(x))-H(D\Phi _{t}^{-1}(x))| \\
&&+|H((D\Phi _{t}^{\epsilon })^{-1}(x))||\mathbf{B}_{0}^{\epsilon }((\Phi
_{t}^{\epsilon })^{-1}(x))-\mathbf{B}_{0}((\Phi _{t}^{\epsilon })^{-1}(x))|
\\
&&+|H((D\Phi _{t}^{\epsilon })^{-1}(x))||\mathbf{B}_{0}((\Phi _{t}^{\epsilon
})^{-1}(x))-\mathbf{B}_{0}(\Phi _{t}^{-1}(x))|
\end{eqnarray*}%
so, by H\"{o}lder inequality, we get 
\begin{equation}
\lefteqn{E[\left\vert \mathbf{B}^{\epsilon }(t,x)-\mathbf{B}(t,x)\right\vert
^{m}]}  \label{Meps-M}
\end{equation}%
\begin{eqnarray*}
&\leq &C\Vert \mathbf{B}_{0}^{\epsilon }\Vert _{0}E[|H((D\Phi _{t}^{\epsilon
})^{-1}(x))-H(D\Phi _{t}^{-1}(x))|^{m}] \\
&&+CE[|H(D\Phi _{t}^{-1}(x))|^{m}]\Vert \mathbf{B}_{0}^{\epsilon }-\mathbf{B}%
_{0}\Vert _{0} \\
&&+CE[|H(D\Phi _{t}^{-1}(x))|^{2m}]^{1/2}E[|\mathbf{B}_{0}((\Phi
_{t}^{\epsilon })^{-1}(x))-\mathbf{B}_{0}(\Phi _{t}^{-1}(x))|^{2m}]^{1/2}.
\end{eqnarray*}%
We will prove that every term on the right-hand-side of (\ref{Meps-M}) tends
to $0$. First notice that $\Vert \mathbf{B}_{0}^{\epsilon }\Vert _{0}$ and $%
E[|H(D\Phi _{t}^{-1}(x)|^{m}]$ are bounded uniformly in $\epsilon $, for
every $m$, since $H$ is a polynomial function. The convergence of $\Vert 
\mathbf{B}_{0}^{\epsilon }-\mathbf{B}_{0}\Vert _{0}$ is ensured by our
assumptions, that of $|\mathbf{B}_{0}((\Phi _{t}^{\epsilon })^{-1}(x))-%
\mathbf{B}_{0}(\Phi _{t}^{-1}(x))|$ by Theorem \ref{flowdiffeo} and
dominated convergence theorem ($\mathbf{B}_{0}$ is bounded). Also the
convergence of $|H(D(\Phi _{t}^{\epsilon })^{-1}(x))-H(D\Phi _{t}^{-1}(x))|$
in $L^{m}(\Omega )$ is a consequence of Theorem \ref{flowdiffeo} and the
fact that $H$ is a polynomial. To see this in detail, we can write this term
as 
\begin{align*}
\lefteqn{H(D(\Phi _{t}^{\epsilon })^{-1}(x))-H(D\Phi _{t}^{-1}(x))} \\
& =\sum_{i,j=1}^{3}\left( \int_{0}^{1}\frac{\partial H}{\partial x_{ij}}(\xi
D(\Phi _{t}^{\epsilon })^{-1}(x)+(1-\xi )D\Phi _{t}^{-1}(x))d\xi \right)
\left( D(\Phi _{t}^{\epsilon })^{-1}(x)-D\Phi _{t}^{-1}(x)\right) _{ij}.
\end{align*}%
The function $H^{\prime }$ is linear (because $H$ is quadratic), so we can
use H\"{o}lder inequality and get 
\begin{align*}
\lefteqn{E[|H((D\Phi _{t}^{\epsilon })^{-1}(x))-H(D\Phi _{t}^{-1}(x))|^{m}]}
\\
& \leq CE[|(D\Phi _{t}^{\epsilon })^{-1}(x)|^{2m}+|D\Phi
_{t}^{-1}(x)|^{2m}]^{1/2}E[|(D\Phi _{t}^{\epsilon })^{-1}(x)-D\Phi
_{t}^{-1}(x)|^{2m}]^{1/2}.
\end{align*}%
The term $E[|(D\Phi _{t}^{\epsilon })^{-1}(x)|^{2m}+|D\Phi
_{t}^{-1}(x)|^{2m}]$ is uniformly bounded in $\epsilon $, thanks to Theorem %
\ref{flowdiffeo}, and the term $E[|(D\Phi _{t}^{\epsilon })^{-1}(x)-D\Phi
_{t}^{-1}(x)|^{2m}]$ tends to $0$ by Theorem \ref{flowdiffeo}. Putting all
together, we get convergence of $\mathbf{B}^{\epsilon }(t,x)$ to $\mathbf{B}%
(t,x)$ in $L^{m}(\Omega )$.

Now, with the help of this convergence, we may prove that $\mathbf{B}$
solves equation (\ref{vectadv_distrib}). We know that (\ref{vectadv_distrib}%
) is satisfied by $\mathbf{B}^{\epsilon }$, pointwise and thus in the
distributional form (formula (\ref{vectadv_distrib})), by integration by
parts. Let us prove that, for every $\varphi $ in $C_{c}^{\infty }(\mathbb{R}%
^{3};\mathbb{R}^{3})$, for every $t$, every term of (\ref{vectadv_distrib})
for $\mathbf{B}^{\epsilon }$ converges in $L^{m}(\Omega )$,for any fixed
finite $m$, to the corresponding term for $\mathbf{B}$. We will use the
previous convergence result and the uniform estimate 
\begin{equation}
\sup_{\epsilon }\sup_{t\in \lbrack 0,T]}\sup_{\left\vert x\right\vert \leq
R}E[|\mathbf{B}^{\epsilon }(t,x)|^{m}]<+\infty ,\ \ \sup_{t\in \lbrack
0,T]}\sup_{\left\vert x\right\vert \leq R}E[|\mathbf{B}(t,x)|^{m}]<+\infty
\end{equation}%
which again follows from Theorem \ref{flowdiffeo}. Take the term $\langle 
\mathbf{B}_{t},\varphi \rangle $. Since 
\begin{equation}
E[|\langle \mathbf{B}_{t}^{\epsilon }-\mathbf{B}_{t},\varphi \rangle
|^{m}]\leq C\langle E[|\mathbf{B}_{t}^{\epsilon }-\mathbf{B}%
_{t}|^{m}],|\varphi |^{m}\rangle ,
\end{equation}%
(in the last term $\left\langle .,.\right\rangle $ denotes the scalar
product in $L^{2}\left( \mathbb{R}^{3}\right) $ between real-valued
functions, not vector fields as usual), $\mathbf{B}^{\epsilon }(t,x)$ tends
to $\mathbf{B}(t,x)$ in $L^{m}(\Omega )$ for every $x$ and $E[|\mathbf{B}%
_{t}^{\epsilon }|^{m}+|\mathbf{B}_{t}|^{m}]$ is bounded uniformly in $%
\epsilon $ and $x$, the convergence of this term follows from dominated
convergence theorem. Similarly one can prove the convergence of the terms $%
\int_{0}^{t}\langle \mathbf{B}_{r},\Delta \varphi \rangle dr$ and $%
\int_{0}^{t}\langle (D\varphi )e_{k},\mathbf{B}_{r}\rangle dW_{r}^{k}$, $%
k=1,2,3$, the last ones using Burkholder inequality 
\begin{equation}
E\left[ \left\vert \int_{0}^{t}\langle (D\varphi )e_{k},\mathbf{B}%
_{r}^{\epsilon }-\mathbf{B}_{r}\rangle dW_{r}^{k}\right\vert ^{m}\right]
\leq C\int_{0}^{t}\langle |(D\varphi )e_{k}|^{m},E[|\mathbf{B}_{r}^{\epsilon
}-\mathbf{B}_{r}|^{m}]\rangle dr.
\end{equation}%
For the last term, $\int_{0}^{t}\langle (D\varphi )^{A}\mathbf{v}_{r},%
\mathbf{B}_{r}\rangle dr$, we have 
\begin{align*}
\lefteqn{E\left[ \left\vert \int_{0}^{t}\left\langle (D\varphi )^{A}\mathbf{v%
}_{r}^{\epsilon },\mathbf{B}_{r}^{\epsilon }\right\rangle
dr-\int_{0}^{t}\left\langle (D\varphi )^{A}\mathbf{v}_{r},\mathbf{B}%
_{r}\right\rangle dr\right\vert ^{m}\right] } \\
& \leq C\int_{0}^{t}\langle |D\varphi |^{m}|\mathbf{v}_{r}^{\epsilon }-%
\mathbf{v}_{r}|^{m},E[|\mathbf{B}_{r}^{\epsilon }|^{m}]\rangle
dr+C\int_{0}^{t}\langle |D\varphi |^{m}|\mathbf{v}_{r}|^{m},E[|\mathbf{B}%
_{r}^{\epsilon }-\mathbf{B}_{r}|^{m}]\rangle dr.
\end{align*}%
Both the two addends in the right-hand-side of this equation tend to $0$ by
dominated convergence theorem, because $\mathbf{v}^{\epsilon }\rightarrow 
\mathbf{v}$ and $E[|\mathbf{B}_{r}^{\epsilon }-\mathbf{B}_{r}|^{m}]%
\rightarrow 0$ for every $(t,x)$ and $|\mathbf{v}^{\epsilon }|+|\mathbf{v}|$%
, $E[|\mathbf{B}_{t}^{\epsilon }|^{m}+|\mathbf{B}_{t}|^{m}]$ are uniformly
bounded. Since all the terms of (\ref{vectadv_distrib}) converge, (\ref%
{vectadv_distrib}) holds for $\mathbf{B}$. Thus $\mathbf{B}$ solves (\ref%
{vectadv}) in the sense of distributions.

\textbf{Step 3 }(property (ii)). Concerning property (ii) of Definition \ref%
{def distrib sol}, we will prove that, for every $t$, for every $\varphi $
in $C_{c}^{\infty }(\mathbb{R}^{3};\mathbb{R}^{3})$, 
\begin{equation}
E\left[ \int_{\mathbb{R}^{3}}\mathbf{B}_{t}\cdot \nabla \varphi dx\right] =0.
\label{Mdivfree}
\end{equation}%
Since, for a.e.\ $\omega $, $\mathbf{B}$ is continuous in $(t,x)$ and since $%
C_{c}^{\infty }(\mathbb{R}^{3};\mathbb{R}^{3})$ is separable, (\ref{Mdivfree}%
) implies that, outside a negligible set in $\Omega $, $B_{t}$ is
divergence-free for every $t$. We know that (\ref{Mdivfree}) is satisfied
for $\mathbf{B}^{\epsilon }$ and that $\mathbf{B}^{\epsilon }$ tends to $%
\mathbf{B}$ in $L^{m}(\Omega )$, for every $m$, with $L^{m}$-norm bounded
uniformly in $x$ (in a ball). Then, applying dominated convergence theorem
as before, we get (\ref{Mdivfree}). The proof is complete.
\end{proof}

Finally, let us prove that the solution given by the previous theorem is
unique.

\begin{lemma}
\label{lemma uniqueness}Let $\mathbf{B}_{0}$ be divergence-free and in $C(%
\mathbb{R}^{3};\mathbb{R}^{3})$ and suppose Condition \ref{Holdercond} hold.
The there is at most one continuous weak solution to equation (\ref{vectadv}%
), given by formula (\ref{reprform 3}).
\end{lemma}

\begin{proof}
\textbf{Step 1} (origin of the proof). Since the equation is linear, it is
sufficient to consider the case $\mathbf{B}_{0}=0$ and prove that, if $%
\mathbf{B}$ is a continuous weak solution to equation (\ref{vectadv}) with $%
\mathbf{B}_{0}=0$, then $\mathbf{B}=0$.

In Proposition \ref{regcase} we proved, by It\^{o}-Kunita-Wentzell formula,
that a regular solution $\mathbf{B}$ satisfies the identity 
\begin{equation*}
\frac{d}{dt}[\mathbf{B}(t,\Phi _{t}(x))]=\left( \mathbf{B}(t,\Phi
_{t}(x))\cdot \nabla \right) {\mathbf{v}}(t,\Phi _{t}(x))
\end{equation*}%
and thus, by uniqueness for equation (\ref{derivSDE}), we got $\mathbf{B}%
(t,\Phi _{t}(x))=D\Phi _{t}(x)\mathbf{B}_{0}(x)$, namely $\mathbf{B}(t,\Phi
_{t}(x))=0$ in the present case (hence $\mathbf{B}=0$). We may also go
further and drop the step involving equation (\ref{derivSDE}): it is
sufficient to differentiate $\left( D\Phi _{t}(x)\right) ^{-1}\mathbf{B}%
(t,\Phi _{t}(x))$:%
\begin{equation*}
\frac{d}{dt}\left[ \left( D\Phi _{t}(x)\right) ^{-1}\mathbf{B}(t,\Phi
_{t}(x))\right] =0
\end{equation*}%
which readily implies $\left( D\Phi _{t}(x)\right) ^{-1}\mathbf{B}(t,\Phi
_{t}(x))=\mathbf{B}_{0}(x)=0$, hence $\mathbf{B}(t,\Phi _{t}(x))=0$ and thus 
$\mathbf{B}=0$. We have used the fact that%
\begin{equation*}
\frac{d}{dt}\left( D\Phi _{t}(x)\right) ^{-1}=-\left( D\Phi _{t}(x)\right)
^{-1}D{\mathbf{v}}(t,\Phi _{t}(x))
\end{equation*}%
which comes from the computation (in the regular case)%
\begin{eqnarray*}
\frac{d}{dt}\left( D\Phi _{t}(x)\right) ^{-1} &=&\lim_{h\rightarrow 0}\frac{%
\left( D\Phi _{t+h}(x)\right) ^{-1}-\left( D\Phi _{t}(x)\right) ^{-1}}{h} \\
&=&\lim_{h\rightarrow 0}\left( D\Phi _{t+h}(x)\right) ^{-1}\frac{\left(
D\Phi _{t}(x)-D\Phi _{t+h}(x)\right) }{h}\left( D\Phi _{t}(x)\right) ^{-1} \\
&=&-\left( D\Phi _{t}(x)\right) ^{-1}\frac{d}{dt}D\Phi _{t}(x)\left( D\Phi
_{t}(x)\right) ^{-1} \\
&=&-\left( D\Phi _{t}(x)\right) ^{-1}D\mathbf{v}(t,\Phi _{t}(x)).
\end{eqnarray*}

These are proofs of uniqueness for regular solutions. If $\mathbf{B}$\ is
only a continuous weak solution, It\^{o}-Kunita-Wentzell formula cannot be
applied. Moreover, $D${$\mathbf{v}$} is a distribution, hence everywhere it
enters the computations it may cause troubles (for instance, the meaning of
equation (\ref{derivSDE}) is less clear; although in mild form it is
meaningful because $D\Phi _{t}(x)$, which exists also in the non-regular
case, is formally its fundamental solution).

Thus we regularize both $\mathbf{B}$ and the flow $\Phi _{t}(x)$. Usually,
with this procedure, the regularized field $\mathbf{B}^{\epsilon }$
satisfies an equation similar to (\ref{derivSDE}) but with a remainder, a
commutator; this has been a successful procedure for linear transport
equations with non-smooth coefficients, see \cite{DiPernaLions};\ in the
stochastic case one has a commutator composed with the flow and the approach
works again well due to variants of the commutator lemma, see \cite%
{FlaGubPri}. The commutator estimates are the central tool in this approach,
both deterministic and stochastic. When special cancellations apply, in
particular due to divergence free conditions, it is possible to follow an
interesting variant of this approach, not based on commutator estimates,
developed by \cite{NevesOliv}. We follow this approach and exploit special
cancellations; in absence of them, the vectorial case proper of this paper
could not be treated (see below the argument about second space derivatives
of the flow).

\textbf{Step 2} (approximation). Let $\rho $ be a $C^{\infty }$ compactly
supported even function on $\mathbb{R}^{3}$ and define the approximations of
identity as $\rho _{\epsilon }(x):=\epsilon ^{-3}\rho (\epsilon ^{-1}x)$,
for $\epsilon >0$. Call $\mathbf{B}^{\epsilon }=\mathbf{B}\ast \rho
_{\epsilon }$, $\mathbf{v}^{\epsilon }=\mathbf{v}\ast \rho _{\epsilon }$
(and similarly for other fields). Then, using $\rho ^{\epsilon }$ as test
function, we get the following equation for $\mathbf{B}^{\epsilon }$,
satisfied pointwise (actually, for a.e.\ $\omega $, for every $(t,x)$, up to
a suitable modification): 
\begin{eqnarray*}
\mathbf{B}^{\epsilon }(t,x) &=&\int_{0}^{t}[(\left( {\mathbf{B}}\cdot \nabla
\right) \mathbf{v})^{\epsilon }(r,x)-(\left( {\mathbf{v}}\cdot \nabla
\right) \mathbf{B})^{\epsilon }(r,x)]dr \\
&&-\sigma \sum_{k=1}^{3}\int_{0}^{t}\partial _{k}\mathbf{B}^{\epsilon
}(r,x)dW_{r}^{k}+\frac{\sigma ^{2}}{2}\int_{0}^{t}\Delta \mathbf{B}%
^{\epsilon }(r,x)dr
\end{eqnarray*}%
where we have used the fact that $\mathbf{B}_{0}^{\epsilon }=0$. Let $\Phi
_{t}^{\epsilon }(x)$ be the regular flow associated to $\mathbf{v}^{\epsilon
}$. Since $\mathbf{B}^{\epsilon }$ is regular, we can now apply It\^{o}%
-Kunita-Wentzell formula to $\mathbf{B}^{\epsilon }(t,\Phi _{t}^{\epsilon
}(x))$ and get (as in the proof of Proposition \ref{regcase}):%
\begin{eqnarray*}
d[\mathbf{B}^{\epsilon }(t,\Phi _{t}^{\epsilon }(x))] &=&\left( d\mathbf{B}%
^{\epsilon }\right) (t,\Phi _{t}^{\epsilon }(x)) \\
&&+\sum_{i=1}^{3}\partial _{x_{i}}\mathbf{B}^{\epsilon }(t,\Phi
_{t}^{\epsilon }(x))d\left( \Phi _{t}^{\epsilon }\right) ^{i}(x) \\
&&+\frac{\sigma ^{2}}{2}\Delta \mathbf{B}^{\epsilon }(t,\Phi _{t}^{\epsilon
}(x))dt-\sigma ^{2}\Delta \mathbf{B}^{\epsilon }(t,\Phi _{t}^{\epsilon
}(x))dt
\end{eqnarray*}%
\begin{eqnarray*}
&=&\left[ (\left( {\mathbf{B}}\cdot \nabla \right) \mathbf{v})^{\epsilon
}-(\left( {\mathbf{v}}\cdot \nabla \right) \mathbf{B})^{\epsilon }(t,\Phi
_{t}^{\epsilon }(x))\right] dt+\frac{\sigma ^{2}}{2}\Delta \mathbf{B}%
^{\epsilon }(t,\Phi _{t}^{\epsilon }(x))dt-\sigma \sum_{k=1}^{3}\partial _{k}%
\mathbf{B}^{\epsilon }(t,\Phi _{t}^{\epsilon }(x))dW_{t}^{k} \\
&&+\left( \mathbf{v}^{\epsilon }\cdot \nabla \right) \mathbf{B}^{\epsilon
}(t,\Phi _{t}^{\epsilon }(x))dt+\sigma \sum_{k=1}^{3}\partial _{k}\mathbf{B}%
^{\epsilon }(t,\Phi _{t}^{\epsilon }(x))dW_{t}^{k} \\
&&-\frac{\sigma ^{2}}{2}\Delta \mathbf{B}^{\epsilon }(t,\Phi _{t}^{\epsilon
}(x))dt
\end{eqnarray*}%
\begin{equation*}
=\left( \left( {\mathbf{B}}\cdot \nabla \right) \mathbf{v})^{\epsilon
}-(\left( {\mathbf{v}}\cdot \nabla \right) \mathbf{B})^{\epsilon }\right)
(t,\Phi _{t}^{\epsilon }(x))dt+\left( \mathbf{v}^{\epsilon }\cdot \nabla
\right) \mathbf{B}^{\epsilon }(t,\Phi _{t}^{\epsilon }(x))dt.
\end{equation*}%
Since $\frac{d}{dt}\left( D\Phi _{t}^{\epsilon }(x)\right) ^{-1}=-\left(
D\Phi _{t}^{\epsilon }(x)\right) ^{-1}D\mathbf{v}^{\epsilon }(t,\Phi
_{t}^{\epsilon }(x))$, we get%
\begin{eqnarray*}
&&\frac{d}{dt}\left[ \left( D\Phi _{t}^{\epsilon }(x)\right) ^{-1}\mathbf{B}%
^{\epsilon }(t,\Phi _{t}^{\epsilon }(x))\right] \\
&=&\left( D\Phi _{t}^{\epsilon }(x)\right) ^{-1}\left[ (\left( {\mathbf{B}}%
\cdot \nabla \right) \mathbf{v})^{\epsilon }-(\left( {\mathbf{v}}\cdot
\nabla \right) \mathbf{B})^{\epsilon }+\left( \mathbf{v}^{\epsilon }\cdot
\nabla \right) \mathbf{B}^{\epsilon }-\left( \mathbf{B}^{\epsilon }\cdot
\nabla \right) \mathbf{v}^{\epsilon }\right] (t,\Phi _{t}^{\epsilon }(x)).
\end{eqnarray*}

Fix $\varphi $ in $C_{c}^{\infty }(\mathbb{R}^{3};\mathbb{R}^{3})$. We
multiply the previous formula by $\varphi $, integrate in space and change
variable $x=\Phi _{t}^{\epsilon }(x^{\prime })$ recalling that $\Phi
_{t}^{\epsilon }$ is measure preserving:%
\begin{equation*}
\int_{\mathbb{R}^{3}}\left( D\Phi _{t}^{\epsilon }(x)\right) ^{-1}\mathbf{B}%
^{\epsilon }(t,\Phi _{t}^{\epsilon }(x))\varphi \left( x\right) dx
\end{equation*}%
\begin{equation*}
=\int_{0}^{t}\int_{\mathbb{R}^{3}}\left[ (\left( {\mathbf{B}}\cdot \nabla
\right) \mathbf{v})^{\epsilon }-(\left( {\mathbf{v}}\cdot \nabla \right) 
\mathbf{B})^{\epsilon }+\left( \mathbf{v}^{\epsilon }\cdot \nabla \right) 
\mathbf{B}^{\epsilon }-\left( \mathbf{B}^{\epsilon }\cdot \nabla \right) 
\mathbf{v}^{\epsilon }\right] (s,x)\psi ^{\epsilon }\left( s,x\right) dxds
\end{equation*}%
where we have introduced the random field%
\begin{equation*}
\psi ^{\epsilon }\left( s,x\right) :=\left( D\Phi _{s}^{\epsilon }((\Phi
_{s}^{\epsilon })^{-1}\left( x\right) )\right) ^{-1}\varphi \left( (\Phi
_{s}^{\epsilon })^{-1}\left( x\right) \right) .
\end{equation*}%
By integration by parts we get: 
\begin{equation}
\int_{\mathbb{R}^{3}}\left( D\Phi _{t}^{\epsilon }(x)\right) ^{-1}\mathbf{B}%
^{\epsilon }(t,\Phi _{t}^{\epsilon }(x))\varphi \left( x\right)
dx=-\sum_{i,j=1}^{3}\int_{0}^{t}\int_{\mathbb{R}^{3}}\left[
(v^{j}B^{i})^{\epsilon }-v^{\epsilon ,j}B^{\epsilon ,i}\right] (s,x)(D\psi
^{\epsilon })_{ij}^{A}(s,x)dxds.  \label{to be converged}
\end{equation}

\textbf{Step 3} (support and convergence of $\psi ^{\epsilon }$). In the
next step we need a technical fact about the support of $x\mapsto \psi
^{\epsilon }\left( t,x,\omega \right) $. Let $R^{\prime }>0$ be such that
the support of $\varphi $ is contained in $B\left( 0,R^{\prime }\right) $.
Define $R^{\epsilon }\left( \omega \right) $ as%
\begin{equation*}
R_{t}^{\epsilon }\left( \omega \right) =\max_{x\in B\left( 0,R^{\prime
}\right) }\left\vert \Phi _{t}^{\epsilon }(x,\omega )\right\vert .
\end{equation*}%
Then the support of $x\mapsto \psi ^{\epsilon }\left( t,x,\omega \right) $
is contained in $\overline{B\left( 0,R_{t}^{\epsilon }\left( \omega \right)
\right) }$. We have 
\begin{equation*}
\Phi _{t}^{\epsilon }(x,\omega )=x+\int_{0}^{t}\mathbf{v}^{\epsilon }\left(
s,\Phi _{s}^{\epsilon }(x,\omega )\right) ds+\sigma \mathbf{W}_{t}\left(
\omega \right)
\end{equation*}%
and there is a constant $C>0$ such that $\left\vert \mathbf{v}^{\epsilon
}\left( s,\Phi _{s}^{\epsilon }((\Phi _{t}^{\epsilon })^{-1}\left( x,\omega
\right) ,\omega )\right) \right\vert \leq C$; thus%
\begin{equation*}
\left\vert \Phi _{t}^{\epsilon }(x,\omega )\right\vert \leq \left\vert
x\right\vert +Ct+\sigma \max_{t\in \left[ 0,T\right] }\left\vert \mathbf{W}%
_{t}\left( \omega \right) \right\vert .
\end{equation*}%
It implies that%
\begin{equation*}
R_{t}^{\epsilon }\left( \omega \right) \leq \overline{R}\left( \omega
\right) :=R^{\prime }+CT+\sigma \max_{t\in \left[ 0,T\right] }\left\vert 
\mathbf{W}_{t}\left( \omega \right) \right\vert
\end{equation*}%
for all $\epsilon >0$, $t\in \left[ 0,T\right] $. The r.v. $\overline{R}%
\left( \omega \right) $ is finite a.s. and thus we have proved that the
function $x\mapsto \psi ^{\epsilon }\left( t,x,\omega \right) $ has a random
support which is contained in $B\left( 0,\overline{R}\left( \omega \right)
\right) $ for all $\epsilon >0$, $t\in \left[ 0,T\right] $, with probability
one. The same result is true replacing $\Phi _{t}^{\epsilon }(x,\omega )$
with $\Phi _{t}(x,\omega )$.

About the convergence of $\psi ^{\epsilon }$, we shall use the following
fact: for a.e.\ $\omega $, possibly passing to a subsequence, $\psi
^{\epsilon }\left( t,\cdot ,\omega \right) $ tends to $\psi \left( t,\cdot
,\omega \right) $ in $L_{loc}^{m}(\mathbb{R}^{3})$ and $\psi ^{\epsilon
}\left( \cdot ,\cdot ,\omega \right) $ tends to $\psi \left( \cdot ,\cdot
,\omega \right) $ in $L_{loc}^{m}([0,T]\times \mathbb{R}^{3})$, for every
finite $m$. Indeed, first notice that $\left( D\Phi _{s}^{\epsilon }((\Phi
_{s}^{\epsilon })^{-1}\left( x\right) )\right) ^{-1}=D(\Phi _{s}^{\epsilon
})^{-1}\left( x\right) $, so that 
\begin{equation*}
\psi ^{\epsilon }\left( s,x\right) =D(\Phi _{s}^{\epsilon })^{-1}\left(
x\right) \varphi \left( (\Phi _{s}^{\epsilon })^{-1}\left( x\right) \right) .
\end{equation*}%
By Theorem \ref{flowdiffeo} and standard arguments like in the proof of
Lemma \ref{lemma_reprformula}, Step 2, $\psi ^{\epsilon }\left( t,\cdot
,\cdot \right) $ converges in $L^{m}(B_{R}\times \Omega )$ for every finite $%
m$ and every $R>0$; this implies that, for a.e.\ $\omega $, possibly passing
to a subsequence, for a.e.\ $\omega $ it converges in $L^{m}(B_{R})$ for
every finite $m$ and every $R>0$; by a diagonal procedure we can choose this
subsequence independently of $m$ and $R$. The proof of the convergence of $%
\psi ^{\epsilon }\left( \cdot ,\cdot ,\omega \right) $ in $%
L_{loc}^{m}([0,T]\times \mathbb{R}^{3})$ is similar.

\textbf{Step 4} (passage to the limit). Now we fix $t>0$ and let $\epsilon $
go to $0$ in formula (\ref{to be converged}). We will prove we obtain in the
limit 
\begin{equation}
\int_{\mathbb{R}^{3}}\left( D\Phi _{t}(x)\right) ^{-1}\mathbf{B}(t,\Phi
_{t}(x))\varphi \left( x\right) dx=0  \label{equivRF}
\end{equation}%
which implies $\mathbf{B}=0$ as already explained above.

The term on the left-hand-side of (\ref{to be converged}) converges,
possibly up to subsequences, to the one on the left-hand-side of (\ref%
{equivRF}). Indeed, by the change variable $x=\Phi _{t}^{\epsilon
}(x^{\prime })$ and the support result of the previous step we have (recall
that $\overline{R}$ is random but independent of $\epsilon >0$) 
\begin{eqnarray*}
\int_{\mathbb{R}^{3}}\left( D\Phi _{t}^{\epsilon }(x^{\prime })\right) ^{-1}%
\mathbf{B}^{\epsilon }(t,\Phi _{t}^{\epsilon }(x^{\prime }))\varphi \left(
x^{\prime }\right) dx^{\prime } &=&\int_{\mathbb{R}^{3}}\mathbf{B}^{\epsilon
}(t,x)\psi ^{\epsilon }\left( t,x\right) dx \\
&=&\int_{B\left( 0,\overline{R}\right) }\mathbf{B}^{\epsilon }(t,x)\psi
^{\epsilon }\left( t,x\right) dx
\end{eqnarray*}%
With probability one, for every $R>0$ the function $\mathbf{B}^{\epsilon
}(t,x)$ converges to $\mathbf{B}(t,x)$ uniformly on $\left[ 0,T\right]
\times B\left( 0,R\right) $, by classical mollifiers arguments. We have seen
in Step 3 that, for a.e.\ $\omega $, possibly passing to a subsequence, $%
\psi ^{\epsilon }\left( t,\cdot ,\omega \right) $ tends to $\psi \left(
t,\cdot ,\omega \right) $ in $L_{loc}^{1}(\mathbb{R}^{3})$. Hence we may
pass to the limit in $\int_{B\left( 0,\overline{R}\right) }\mathbf{B}%
^{\epsilon }(t,x)\psi ^{\epsilon }\left( t,x\right) dx$, for a.e.\ $\omega $%
; the limit is $\int_{B\left( 0,\overline{R}\right) }\mathbf{B}(t,x)\psi
\left( t,x\right) dx$ which gives the left-hand-side of (\ref{equivRF}) by
going backwards with the same computations.

Let us consider now the term on the right-hand-side of (\ref{to be converged}%
); we want to prove that it converges to zero. It is not difficult to show
that, for a.e.\ $\omega $, both $(v^{j}B^{i})^{\epsilon }$ and $v^{\epsilon
,j}B^{\epsilon ,i}$ converge to $v^{j}B^{i}$ in $C([0,T]\times \mathbb{R}%
^{3})$ (namely, uniformly on compact sets) so $(v^{j}B^{i})^{\epsilon
}-v^{\epsilon ,j}B^{\epsilon ,i}$ tends to $0$ in that space. The term $%
(D\psi ^{\epsilon })_{ij}^{A}(s,x)$ could look problematic at a first view,
since it seems to involve the second derivatives of the flow $\Phi
^{\epsilon }$, which are not under control. But this is not the case,
because we only need the antisymmetric part of the derivative. Indeed,
differentiating $\psi ^{\epsilon }=(D((\Phi ^{\epsilon })^{-1}))^{T}\varphi
((\Phi ^{\epsilon })^{-1})$, we get 
\begin{equation*}
(D\psi ^{\epsilon })_{ij}=\sum_{k=1}^{3}\partial _{j}\partial _{i}((\Phi
^{\epsilon })^{-1})^{k}\varphi _{k}((\Phi ^{\epsilon
})^{-1})+\sum_{k=1}^{3}\partial _{i}((\Phi ^{\epsilon })^{-1})^{k}\partial
_{i}[\varphi _{k}((\Phi ^{\epsilon })^{-1})].
\end{equation*}%
The possible problem is only with the first addend. Its antisymmetric part
however is 
\begin{equation*}
\sum_{k=1}^{3}\partial _{j}\partial _{i}((\Phi ^{\epsilon
})^{-1})^{k}\varphi _{k}((\Phi ^{\epsilon })^{-1})-\sum_{k=1}^{3}\partial
_{i}\partial _{j}((\Phi ^{\epsilon })^{-1})^{k}\varphi _{k}((\Phi ^{\epsilon
})^{-1})=0.
\end{equation*}%
So $(D\psi ^{\epsilon })^{A}$ involves only powers of first derivatives of $%
\Phi ^{\epsilon }$. Hence, using again arguments like in proof of Lemma \ref%
{lemma_reprformula}, up to subsequences, $(D\psi ^{\epsilon })^{A}$
converges to $(D\psi )^{A}$ in $L_{loc}^{1}([0,T]\times B_{R})$, with
probability one. Using again the uniform random support of $\psi ^{\epsilon
} $ we see that the term on the right-hand-side of (\ref{to be converged}),
equal to 
\begin{equation*}
-\sum_{i,j=1}^{3}\int_{0}^{t}\int_{B\left( 0,\overline{R}\right) }\left[
(v^{j}B^{i})^{\epsilon }-v^{\epsilon ,j}B^{\epsilon ,i}\right] (s,x)(D\psi
^{\epsilon })_{ij}^{A}(s,x)dxds
\end{equation*}%
converges to zero, with probability one. Then (\ref{equivRF}) is proved and
the proof is complete.
\end{proof}

\end{document}